\documentclass[a4paper, 11pt]{article}

\usepackage[top=1in, bottom=1in, left=1in, right=1in]{geometry}
\usepackage{amsmath}
\usepackage{amssymb}
\usepackage{enumerate}
\usepackage{amsthm}

\newtheorem{Theorem}{Theorem}[subsection]
\newtheorem{Definition}[Theorem]{Definition}
\newtheorem{Remark}[Theorem]{Remark}
\newtheorem{Proposition}[Theorem]{Proposition}
\newtheorem{Axiom}[Theorem]{Axiom}
\newtheorem{Example}[Theorem]{Example}

\begin{document}

\title{\textbf{Set Matrix Theory as a Physically Motivated Generalization of Zermelo-Fraenkel Set Theory}}

\author{
        Marcoen J.T.F. Cabbolet$^{1,2,}$\footnote{e-mail: Marcoen.Cabbolet@vub.ac.be}\ , \ Harrie C.M. de Swart$^{3,}$\footnote{e-mail: deswart@fwb.eur.nl}\\
        \\
        \small{\textit{$^1$Center for Logic and Philosophy of Science, Vrije Universiteit Brussel}}\\
        \small{\textit{Pleinlaan 2, 1050 Brussels (Belgium)}}\\
        \small{\textit{$^2$Institute of Theoretical Physics, Kharkov Institute of Physics and Technology}}\\
        \small{\textit{Akademicheskaya str. 1, 61108 Kharkov (Ukraine)}}\\
        \small{\textit{$^3$Faculty of Philosophy, Erasmus University}}\\
        \small{\textit{PO Box 1738, 3000 DR Rotterdam (the Netherlands)}}
        }

\maketitle

\begin{abstract}
Recently, the Elementary Process Theory (EPT) has been developed as a set of fundamental principles that might underlie a gravitational repulsion of matter and antimatter. This paper presents set matrix theory (SMT) as the foundation of the mathematical-logical framework in which the EPT has been formalized: it is, namely, objectionable to use Zermelo-Fraenkel set theory (ZF) as such. SMT is a generalization of ZF: whereas ZF uses only sets as primitive objects, in the framework of SMT finite matrices with set-valued entries are objects \emph{sui generis}, with a $1\times1$ set matrix $[x]$ being identical to the set $x$. It is proved that every set that can be constructed in ZF can also be constructed in SMT: as a mathematical foundation, SMT is thus not weaker than ZF. In addition, it is shown that SMT is more suitable than ZF for the intended application to physics. The conclusion is that SMT, contrary to ZF, is acceptable as the mathematical-logical foundation of the framework for physics that is determined by the EPT.
\end{abstract}

\section{Introduction}
It is a mathematical fact that Zermelo-Fraenkel set theory (ZF) can be used as foundation for virtually all of modern mathematics: this paper is best started by emphasizing that this fact is left \emph{unquestionable} in the remainder of the text -- there is no \emph{mathematical} motivation to replace ZF.



%

Recently, however, the Elementary Process Theory (EPT) has been developed as a formal axiomatic system that can be applied as a foundational framework for physics under the condition that matter and antimatter repulse each other gravitationally, cf. \cite{bib:Cabb, bib:Cabb2, bib:Cabb3}. Currently there is no proof that such repulsion exists, but on the other hand there is also no proof that it doesn`t exist: the AEGIS collaboration at CERN aims to establish the coupling of antimatter with the gravitational field of ordinary matter; results are expected in 2014/2015\footnote{M. Doser, CERN, personal communication (2011).}. The work on the EPT entailed a search for first principles that are consistent with the main consequence of the assumed existence of gravitational repulsion, being -- as shown in \cite{bib:Cabb3} -- that antimatter then necessarily has positive rest mass and negative gravitational mass as observable properties. This is of fundamental interest, as this combination of properties is \emph{absolutely impossible} in the framework of contemporary physics. Thus speaking, if the AEGIS collaboration establishes that antimatter is repulsed by the gravitational field of the earth, then the contemporary foundations of physics are experimentally falsified and a new foundational framework for physics is then required. In broad lines, that is the motivation for the development of the EPT. This work thus implies the position that it is true that experimental successes compel one to accept that Quantum Theory has merit, but that it nevertheless is \emph{not} the final answer regarding the physics of the microcosmos; it also implies the belief that the key lies in gravitational repulsion: if this is a fact of nature, then the most fundamental principles governing the universe are particularly simple -- as the EPT demonstrates.

The point is now that \emph{any} attempt to formalize the EPT within the framework of ZF, that is, with the usual language of mathematics, is liable to objection. The feature of ZF, namely, that everything has to be a set causes \emph{philosophical} -- not \emph{mathematical}! -- difficulties that are both unavoidable and unsolvable within the framework of ZF\footnote{This work is \emph{interdisciplinary}: it finds itself at the intersection of mathematics, physics and philosophy!}. There are then precisely two options to choose from:
\begin{enumerate}[(i)]
\item one can maintain ZF, meaning that one has to swallow that the formalization of the EPT is surrounded with difficulties;
\item one can solve the issues with the formalization of the EPT, meaning that one has to reject ZF on \emph{philosophical} grounds.
\end{enumerate}
There is no metaprinciple that \emph{compels} one to choose either one of these options: it is, thus, a free choice. Obviously, the second option is then the more logical one to choose in the research program on the EPT. This choice entails, thus, the view that mathematics \emph{in the first place} is meant to provide a language for the natural sciences: if the language of mathematics (\emph{in casu} that of ZF) fails in its purpose, then it has to be adjusted. This view corresponds with the adage `\emph{mathematica ancilla physicae}' (mathematics is the servant of physics).

With the above subtle motivation,
it was decided to develop a new foundational theory for mathematics; a condition was that the resulting theory should not be weaker than ZF, that is, every set that can be constructed in ZF must also be constructible in the framework of the new theory. It turned out that it was sufficient to generalize ZF to a theory based on matrices of sets instead of sets alone. For that matter, it was decided to merge the primitive notion of a matrix with axiomatic set theory into a new mathematical theory; the resulting theory was called set matrix theory (SMT). The primitive notion used is that of a $m \times n$ matrix, that can be described as an ordered rectangular object, consisting of $m n$ entries $t_{ij}$ arranged evenly spaced in $m$ rows and $n$ columns within square brackets, as in $\left [ \begin{array}{ccc} t_{11} & \ldots & t_{1n} \\ \vdots &   & \vdots \\ t_{m1} & \ldots & t_{mn} \end{array} \right ]$. In SMT, the entries $t_{ij}$ of matrices are allowed to be $p_{ij} \times q_{ij}$ matrices themselves, but in the end every matrix has to consist of a finite number of simple entries (sets). Axioms were identified for the matrices, as well as for sets. The idea was to describe sets axiomatically in such a way, that matrices could be elements of sets. For that matter, generalizations of the axioms of ZF, given e.g. in \cite{bib:Dalen}, could be used.

The remainder of this paper is organized as follows: section 2 presents the arguments against ZF in detail; section 3 introduces SMT axiomatically, and section 4 demonstrates (i) that SMT is not weaker than ZF as a foundational theory for mathematics, and (ii) that SMT solves the issues with the formalization of the EPT. Conclusions are given in section 5.

\section{The arguments against ZF}
In the aforementioned research that led to the development of the EPT, the assumption that gravitational repulsion exists was not based on empirical data, but on what Descartes called \emph{une id\'{e}e claire et distincte}: an idea that presents itself so clearly and distinctively to the mind that there is no reason to doubt it. This idea cannot be expressed easily in usual language, but its essence is captured in these two sentences:
\begin{enumerate}[(i)]
\item	if a coin has fallen down from one's hand onto a table, then in opposite time-direction an anticoin has fallen upwards from the antitable into the antihand\footnote{For comparison: Feynman's interpretation of a positron (i.e. an anti-electron) is, that a positron is an electron traveling backwards in time, cf. \cite{bib:Feyn}.};
\item	this tendency to `fall upwards' is preserved in antimatter that exists in `our' time-direction.
\end{enumerate}
Therefore, in the world view based on the EPT, the physical universe consists of a world and an antiworld; a component of this universe is simultaneously a constituent of a world and a constituent of an antiworld. Thus, in the EPT matrices of the type $\left [ \begin{array}{c} x \\ y \end{array} \right ]$ with set-valued entries $x$ and $y$ (set matrices) are used as designators of components of the physical universe consisting of a constituent $x$ of the world and a constituent $y$ of the antiworld.

The first complication arises from the truth-condition of knowledge, which is an essential aspect of every theory intended as a foundation for physics. For the EPT, as a formalized theory, to represent knowledge of the physical universe, the condition was set that there had to be a \emph{direct} relation between components of the physical universe and the theoretical terms referring to these components: entities that occur in the ontology for physics had to be designated by entities that occur as such in the ontology for mathematics. A conflict then arises from the fact that in the framework of ZF, a matrix can not be considered as something existing in its own right as a square array of entries, because everything has to be a set. Thus, in the framework of ZF a $m \times n$ set matrix has to be formalized as a set, for example, as a function on the cartesian product $\{1, \ldots, m\}\times\{1, \ldots, n\}$. Thus, a $2\times1$ set matrix $\left [ \begin{array}{c} x \\ y \end{array} \right ]$ can be defined as a function $f$, given by the following function prescription:
\begin{equation}\label{eq:1}
f:\langle 1, 1 \rangle \mapsto x
\end{equation}
\begin{equation}\label{eq:2}
f:\langle 2, 1 \rangle \mapsto y
\end{equation}
Using the set-theoretical definition of a function, this function $f$ as a set is thus given by
\begin{equation}\label{eq:3}
f = \{ \langle \langle 1, 1 \rangle, x \rangle , \langle \langle 2, 1 \rangle, y \rangle \}
\end{equation}
The set-theoretical definition of an ordered two-tuple, cf. \cite{bib:Dalen}, is the following:
\begin{equation}\label{eq:4}
\langle a, b \rangle = \{\{a\}, \{a, b\}\}
\end{equation}
Combining (\ref{eq:3}) and (\ref{eq:4}), this gives
\begin{equation}\label{eq:5}
f = \{\{\{\langle 1, 1 \rangle \}, \{\langle 1, 1 \rangle, x\}\}, \{\{\langle 2, 1 \rangle \}, \{\langle 2, 1 \rangle , y\}\}\}
\end{equation}
Concluding, in the framework of ZF, the $2\times1$ set matrix $\left [ \begin{array}{c} x \\ y \end{array} \right ]$ is thus merely the \emph{notation} for the set $f$ in (\ref{eq:5}): the \emph{actual object} in the set-theoretical universe, namely, is $f$. With regard to the intended application as designators of components of the physical universe, obviously this set $f$ is \emph{not} a direct designator of the physical component in question: the two constituents, designated by $x$ and $y$, are not at all designated by $f$ but by elements of elements of $f$. This complication does not disappear by defining a $2\times1$ set matrix $\left [ \begin{array}{c} x \\ y \end{array} \right ]$ otherwise as some set $S$: it remains the case that it is not the \emph{actual mathematical object} (actual because it exists as such in the universe of sets) that designates the physical object; it is merely the \emph{notation} $\left [ \begin{array}{c} x \\ y \end{array} \right ]$ of the mathematical object $S$ that designates the physical object. Thus, the definition of $2\times1$ set matrices $\left [ \begin{array}{c} x \\ y \end{array} \right ]$ as (notations for) sets leads to mathematical designators that are not in a direct relation with the physical objects they designate. To put this in other words: in the context of the EPT, objects that exist in the physical universe cannot be designated by objects that exist in the mathematical universe if matrices have to be defined as sets. This was considered \emph{inappropriate}; note that this is not a mathematical argument against the definition of matrices as sets.

A second complication arises from the maxim that \emph{every} theorem of the formal axiomatic system containing the EPT has to yield a statement about the physical universe -- which is intended to be true -- by applying the interpretation rules. The point here is that the EPT contains bidirectional expressions of the type $\left [ \begin{array}{c} a \\ b \end{array} \right ]: \left [ \begin{array}{c} f \\ g \end{array} \right ] \begin{array}{c} \rightarrow \\ \leftarrow \end{array} \left [ \begin{array}{c} x \\ y \end{array} \right ]$; these are first-order expressions $P\alpha\beta\gamma$ that had to be formalized as well-formed formulas in a mathematical framework. The interpretation rule\footnote{It is emphasized that the words `component', `constituent' and `discrete transition' in this interpretation rule thus all concern the \emph{physical} universe, not the \emph{mathematical} universe.} for such an expression is that the component $\left [ \begin{array}{c} a \\ b \end{array} \right ]$ mediates an equilibrium between the components $\left [ \begin{array}{c} f \\ g \end{array} \right ]$ and $\left [ \begin{array}{c} x \\ y \end{array} \right ]$, which is to say that the constituent $a$ of the world effects a discrete transition in the world from the constituent $f$ to the constituent $x$ while the constituent $b$ of the antiworld effects a discrete transition in the antiworld from the constituent $y$ to the constituent $g$. In other words, one has to think of two simultaneous but oppositely directed discrete transitions. Now let these bidirectional expressions be formalized in ZF, and let the $2\times1$ set matrices $\left [ \begin{array}{c} a \\ b \end{array} \right ]$, $\left [ \begin{array}{c} f \\ g \end{array} \right ]$, and $\left [ \begin{array}{c} x \\ y \end{array} \right ]$ be identical to the sets $S$, $T$, and $V$, respectively. Using substitutivity of equality
\begin{equation}\label{eq:6}
u = t \vdash \Psi(u) \Leftrightarrow \Psi([t\backslash u])
\end{equation}
it follows that in ZF a formula $S:T \rightleftarrows V$ can be derived, as in
\begin{equation}\label{eq:6a}
\left [ \begin{array}{c} a \\ b \end{array} \right ]: \left [ \begin{array}{c} f \\ g \end{array} \right ] \begin{array}{c} \rightarrow \\ \leftarrow \end{array} \left [ \begin{array}{c} x \\ y \end{array} \right ]
\vdash_{ZF}
S:T \rightleftarrows V
\end{equation}
Thus, if the EPT is formalized in ZF, then expressions $S:T \rightleftarrows V$ are theorems of the axiomatic system containing the EPT, but these \emph{cannot} be translated into statements about physical reality because the interpretation rule doesn't apply to such expressions without $2\times1$ set matrices. In other words: a formalization of the EPT in ZF gives rise to what Redhead called \emph{weakly surplus structure} -- the formalism then contains uninterpretable elements \cite{bib:Redhead}. The aforementioned maxim is then unattainable and this was considered \emph{unacceptable}. If, on the other hand, a $2\times1$ set matrix is defined as an object on itself, not identical to any set, then it is not possible to construct such nonsensical expressions of the type $S:T \rightleftarrows V$ from
these expressions of the type $\left [ \begin{array}{c} a \\ b \end{array} \right ]: \left [ \begin{array}{c} f \\ g \end{array} \right ] \begin{array}{c} \rightarrow \\ \leftarrow \end{array} \left [ \begin{array}{c} x \\ y \end{array} \right ]$.

This concludes the exposition on the arguments against ZF.

\section{Axiomatic introduction of SMT}

\subsection{The language of SMT}

\begin{Definition} \label{Def:Symbols} \rm
The vocabulary for SMT is a first order language with identity, and consists of the following symbols:
\begin{enumerate}[(i)]
\item the simple constant $\emptyset$
\item countably many variables ranging over sets, possibly sets of matrices: \\
$x$, $y$, $z$, $\ldots$
\item countably many function symbols:
\begin{itemize}
\item the unary function symbol $f_{1}^{1}$, with $f^{1}_{1} = f_{1 \times 1}$
\item the first binary function symbol $f^{2}_{1}$ with $f^{2}_{1} = f_{1 \times 2}$
\item the second binary function symbol $f^{2}_{2}$ with $f^{2}_{2}= f_{2 \times 1}$
\item the first ternary function symbol $f^{3}_{1}$ with $f^{3}_{1}= f_{1 \times 3}$
\item the second ternary function symbol, $f^{3}_{2}(x, y, z) = f_{1 \times 2}(f_{1 \times 2}(x, y), z)$
\item the third ternary function symbol, $f^{3}_{3}(x, y, z) = f_{1 \times 2}(x, f_{1 \times 2}(y, z))$
\item the fourth ternary function symbol, $f^{3}_{4}(x, y,z) = f_{1 \times 2}(f_{2 \times 1}(x, y), z)$
\item the fifth ternary function symbol, $f^{3}_{5}(x, y, z) = f_{1 \times 2}(x, f_{2 \times 1}(y, z))$\\
and so forth.
\end{itemize}
\item countably many variables ranging over matrices: $\alpha$, $\beta$, $\gamma$ $\ldots$
\item the binary predicate symbols $\in$ and $=$
\item the usual connectives $\neg$, $\Rightarrow$, $\Leftrightarrow$,
$\wedge$, $\vee$
\item the usual quantifiers $\forall$ and $\exists$
\end{enumerate}
\end{Definition}

\begin{Definition} \label{Def:Syntax} \rm
The syntax of the formal language is defined as follows:
\begin{enumerate}[(i)]
\item if $t$ is a simple constant or a variable ranging over sets, then $t$ is a term;
\item if $t_1, \ldots, t_n$ are $n$ terms and $f^n_i$ is an $n$-ary function symbol, then $f^n_i(t_1, \ldots, t_n)$ is a composite term;
\item if $t_1$ and $t_2$ are terms and $P$ is one of the binary predicate letters $\in$ or =, then $t_1 P t_2$ is an
atomic formula (infix notation);
\item if $\Phi$ is a formula, then $\neg \Phi$ is a formula;
\item if $\Phi$ and $\Psi$ are formulas, then ($\Phi \Rightarrow \Psi$), ($\Phi \wedge \Psi$), ($\Phi \vee \Psi$) are formulas;
\item if $\Phi$ is a formula, $Q$ a quantifier $\exists$ or $\forall$, and $x$ a variable ranging over sets, then $Qx(\Phi)$
is a formula;
\item if $\Phi(x)$ is a formula in which the variable $x$ ranging over sets occurs not bounded by a quantifier, $Q$ is a quantifier,
and $\alpha$ is a variable ranging over matrices, then $Q \alpha (\Phi(\alpha))$ is a formula, where $\Phi(\alpha)$ results from
$\Phi(x)$ by replacing $x$ everywhere by $\alpha$.
\end{enumerate}
\end{Definition}

\begin{Remark} \label{Rem:ListExhaustive}\rm
The list in Definition \ref{Def:Symbols}(iii) of function symbols is exhaustive. That is, for every composite term $t$, constructed by applying the clauses \ref{Def:Syntax}(i) and (ii) finitely many times, there is a term $f^{n}_{i}(x_1, \ldots, x_n)$ of which $x_1, \ldots, x_n$ are interpretable as sets, such that $t = f^{n}_{i}(x_1, \ldots, x_n)$. For example, the composite term $f^{2}_{1}(f^{2}_{1}(x, y), z)$ can also be written as the term $f^{3}_{2}(x, y, z)$. This exhaustive enumeration of function symbols is very useful for the formulation of the axioms.
\end{Remark}

\begin{Remark} \label{Rem:Notation} \rm
The following are standard notations for terms (set matrices) and formulas:
\begin{enumerate}[(i)]
\item outer parentheses `(' and `)' can be omitted.
\item $t_1 \not \in t_2$ denotes $\neg$ $t_1 \in t_2$
\item $t_1 \neq t_2$ denotes $\neg$ $t_1 = t_2$
\item $[ x ]$ denotes $f_{1 \times 1}(x)$
\item $[ x \ \ y]$ denotes $f_{1 \times 2}(x, y)$
\item $\left [ \begin{array}{c} x \\ y \end{array} \right ]$ denotes $f_{2 \times 1}(x, y)$
\item $[ x \ \ y \ \ z]$ denotes $f_{1 \times 3}(x, y, z)$
\end{enumerate}
and so forth.
\end{Remark}

\begin{Remark} \label{Rem:SubstitutionRule} \rm
The following substitution rule is logically valid for the variables ranging over matrices:
\begin{equation}\label{eq:7}
\forall \alpha (\Psi (\alpha)) \Rightarrow \Psi (f^{n}_{i}(x_1, \ldots, x_n))
\end{equation}
The substitution rule (\ref{eq:7}) applies to any function symbol $f^{n}_{i}$ and any $n$ sets $x_1, \ldots, x_n$; \linebreak here the formula $\Psi (f^{n}_{i}(x_1, \ldots, x_n))$ is the formula that results from $\Psi (\alpha)$ by replacing $\alpha$ everywhere by $f^{n}_{i}(x_1, \ldots, x_n)$. Note that the variables $\alpha$ ranging over matrices occur only in quantified formulas; using such variables in open formulas is not needed for the present axiomatization (Ockham's razor).
\end{Remark}

\subsection{The axioms of SMT}

\begin{Axiom}\label{Ax:SetMatrixScheme}(Set Matrix Axiom Scheme)\rm:\\
$\forall x_{1} \ldots \forall x_{n} \exists \alpha ( \alpha = f^{n}_{i}(x_{1}, \ldots, x_{n}) )$
\end{Axiom}\noindent
The Set Matrix Axiom Scheme is a countably infinite scheme, consisting of an axiom for every function symbol $f^{n}_{i}$. This axiom scheme guarantees that for any $m\cdot n$ sets $x_{11}, \ldots, x_{mn}$ there is a set matrix $\left [ \begin{array}{ccc} x_{11} & \ldots & x_{1n} \\ \vdots &   & \vdots \\ x_{m1} & \ldots & x_{mn} \end{array} \right ]$ with these sets as entries; in addition, it guarantees that for any $m\cdot n$ set matrices $\alpha_{11}, \ldots, \alpha_{mn}$ there is a set matrix $\left [ \begin{array}{ccc} \alpha_{11} & \ldots & \alpha_{1n} \\ \vdots &   & \vdots \\ \alpha_{m1} & \ldots & \alpha_{mn} \end{array} \right ]$ with these set matrices as entries.

\begin{Axiom}\label{Ax:Reduction}(Reduction Axiom)\rm:\\
$\forall x ( [x] = x )$
\end{Axiom}\noindent
The purpose of the Reduction Axiom is to equate set matrices having one set as sole entry with that set itself. Hence, any set $x$ is identical to the set matrix $[x]$ of dimension one by one. For example, the empty set $\emptyset$ is identical to the one by one set matrix $[\emptyset]$ containing the empty set $\emptyset$ as sole entry.

\begin{Axiom}\label{Ax:OmissionScheme}(Omission Axiom Scheme)\rm:\\
$\forall \alpha_{11} \ldots \forall \alpha_{mn} (\left [ \left [ \begin{array}{ccc} \alpha_{11} & \ldots & \alpha_{1n} \\ \vdots &   & \vdots \\ \alpha_{m1} & \ldots & \alpha_{mn} \end{array} \right ] \right ] = \left [ \begin{array}{ccc} \alpha_{11} & \ldots & \alpha_{1n} \\ \vdots &   & \vdots \\ \alpha_{m1} & \ldots & \alpha_{mn} \end{array} \right ] \ \ \ m\cdot n \geq 2$
\end{Axiom}\noindent
The Omission Axiom Scheme is a countably infinite scheme consisting of an axiom for every function symbol $f_{m \times n}$ with $m \cdot n \geq 2$. The Omission Axiom Scheme is to formalize that a matrix $\left [ \left [ \begin{array}{ccc} t_{11} & \ldots & t_{1n} \\ \vdots &   & \vdots \\ t_{m1} & \ldots & t_{mn} \end{array} \right ] \right ]$, constructed by placing an existing matrix $\left [ \begin{array}{ccc} t_{11} & \ldots & t_{1n} \\ \vdots &   & \vdots \\ t_{m1} & \ldots & t_{mn} \end{array} \right ]$ as sole entry in square brackets `[' and `]', is identical to the existing matrix $\left [ \begin{array}{ccc} t_{11} & \ldots & t_{1n} \\ \vdots &   & \vdots \\ t_{m1} & \ldots & t_{mn} \end{array} \right ]$. It should be noted that this includes the case that the entries $t_{ij}$ are sets: the Reduction Axiom, namely, equates a $1 \times 1$ set matrix $[z]$ with the set $z$. The notion of a matrix is thus different from the notion of a set, where $\{x\} \neq x$ for any set $x$.

\begin{Axiom}\label{Ax:EpsilonScheme}(Epsilon Axiom Scheme)\rm:\\
$\forall \alpha_{11} \ldots \forall \alpha_{mn}\forall \beta_{11} \ldots \forall \beta_{pq} ( \left [ \begin{array}{ccc} \alpha_{11} & \ldots & \alpha_{1n} \\ \vdots &   & \vdots \\ \alpha_{m1} & \ldots & \alpha_{mn} \end{array} \right ]  \not \in \left [ \begin{array}{ccc} \beta_{11} & \ldots & \beta_{1q} \\ \vdots &   & \vdots \\ \beta_{p1} & \ldots & \beta_{pq} \end{array} \right ]) \ \ \ p\cdot q \geq 2$
\end{Axiom}\noindent
The Epsilon Axiom Scheme is a countably infinite axiom scheme consisting of an axiom for every function symbol $f_{m \times n}$  and for every function symbol $f_{p \times q}$ with $p\cdot q \geq 2$.
\begin{Axiom}\label{Ax:DivisionScheme}(Division Axiom Scheme)\rm:\\
$\forall x_1 \ldots \forall x_n \forall y_1 \ldots \forall y_m (f^{n}_{i} (x_1 , \ldots , x_n) \neq f^{m}_{j}(y_1, \ldots , y_m)) \
\ \ for \ n \neq m \vee i \neq j$
\end{Axiom}\noindent
The Division Axiom Scheme is a countably infinite axiom scheme, consisting of an axiom for every choice of different function symbols $f^{n}_{i}$ and $f^{m}_{j}$.

\begin{Remark} \rm
The Epsilon Axiom Scheme together with the Division Axiom Scheme formalize that
\begin{enumerate}[(i)]
\item set matrices, consisting of more than one set, have no elements in the sense of the $\in$-relation;
\item set matrices, consisting of more than one set, are different from any set.
\end{enumerate}
As a consequence, set matrices of other dimensions than $1\times 1$ are objects \emph{sui generis}: contrary to the universe of ZF, the universe of SMT contains, thus, \emph{objects that are not sets}. It will be shown in section 4.2 that this is essential for solving the problems with the formulation of the EPT. So concretely, for any sets $x$ and $y$, the set matrix $ \left [\begin{array} {c} x \\ y \end{array} \right ]$ is different from any $1 \times 1$ set matrix $[ z ]$, which by the Reduction Axiom \ref{Ax:Reduction} is identical to the set $z$.  So concerning the $\in$-relation, only expressions of the type
\begin{equation}
f^{n}_{i}(x_{1}, \ldots, x_{n}) \in y
\end{equation}
with a set $y$ to the right of the $\in$-symbol are contingent.
\end{Remark}

\begin{Remark} \rm
Because matrices are viewed as objects existing in their own right in the framework of SMT, the common notation $\langle x_1, \ldots , x_n \rangle$ for an ordered $n$-tuple of sets can be applied as a special notation for a $1 \times n$ set matrix:
\begin{equation} \label{eq:9}
\langle x_1, \ldots , x_n \rangle := [ x_1 \ \ldots \ x_n ]
\end{equation}
But by the Division Axiom Scheme, for any three sets $x_1$, $x_2$, and $x_3$,
\begin{equation}
[ x_1 \ \ x_2 \ \ x_3] \neq [ [ x_1 \ \ x_2 ] \ \ x_3 ]
\end{equation}
Therefore, accepting a $1 \times n$ matrix with $n$ sets as entries as the definition of an ordered $n$-tuple requires the rejection of the recursive definition of an ordered $n$-tuple of sets $\langle x_1,\ldots , x_n \rangle$
\begin{equation}
\langle x \rangle := x
\end{equation}
\begin{equation}
\langle x_1, \ldots , x_{n+1} \rangle := \langle \langle x_1, \ldots , x_n \rangle , x_{n+1} \rangle
\end{equation}
given in the literature, cf. \cite{bib:Dalen}. In the remainder of this text, equation (\ref{eq:9}) will be used.
\end{Remark}
\begin{Axiom} \label{Ax:ExtensionalityMatrices} (Extensionality Axiom Scheme for Set Matrices)\rm:\\
\begin{flushleft}
$\forall \alpha_{11} \ldots \forall \alpha_{mn}\forall \beta_{11} \ldots \forall \beta_{mn} ( \left [ \begin{array}{ccc} \alpha_{11} & \ldots & \alpha_{1n} \\ \vdots &   & \vdots \\ \alpha_{m1} & \ldots & \alpha_{mn} \end{array} \right ]  = \left [ \begin{array}{ccc} \beta_{11} & \ldots & \beta_{1n} \\ \vdots &   & \vdots \\ \beta_{m1} & \ldots & \beta_{mn} \end{array} \right ] \Rightarrow$
\end{flushleft}
\begin{flushright}
$\alpha_{11} = \beta_{11} \wedge \ldots \wedge \alpha_{mn} = \beta_{mn})$
\end{flushright}
\end{Axiom}\noindent
The Extensionality Axiom Scheme for Set Matrices is a countably infinite scheme, with an axiom for every functional symbol $f_{m \times n}$ with $m\cdot n > 1$. The purpose of the Extensionality Axiom Scheme for Set Matrices is to formalize that two set matrices of the same type are identical if and only if the corresponding entries are identical: this reduces the
identity of matrices to a conjunction of identities of sets.

\begin{Definition} \label{Def:Subset}
$\\ \forall x \forall y ( x \subseteq y \Leftrightarrow \forall \alpha ( \alpha \in x \Rightarrow \alpha \in y))$
\end{Definition}\noindent
The interpretation of this definition is that a set $x$ is a subset of a set $y$  if and only if every matrix that is an element of $x$ is also an element of $y$. Note that this includes the case that $x$ and $y$ are sets of sets: quantification over matrices includes quantification over sets (\emph{vide supra}).

\begin{Axiom} \label{Ax:Exp-ExtensionalitySets} (Generalized Extensionality Axiom for Sets)\rm: \\
$\forall x \forall y ( x = y \Leftrightarrow x \subseteq y \wedge y \subseteq x ) $
\end{Axiom}\noindent
The Generalized Extensionality Axiom for Sets is not exactly the same as the extensionality axiom for sets of ZF, because the definition of $x \subseteq y$ is different in the current framework.

\begin{Axiom} \label{Ax:Exp-Emptiness} (Generalized Axiom of Emptiness)\rm: \\
$\exists x \forall \alpha (\alpha \not \in x)$
\end{Axiom}\noindent
This axiom formalizes that there is a set $x$, such that no matrix is an element of $x$. Suppose, there were two such sets $x$ and $y$.
Then by axiom \ref{Ax:Exp-ExtensionalitySets}, $x = y$. Hence there is precisely one set $x$ such that $\forall \alpha ( \alpha \not \in x)$. This set is called the empty set, denoted by the
constant $\emptyset$ or $\{ \}$, which by the Reduction Axiom \ref{Ax:Reduction} is identical to $[\emptyset]$ or $[\{\}]$.

\begin{Remark} \label{Rem:MatricesNotSets} \rm
As remarked earlier, for all $n > 1$, set matrices $f^n_i(x_1, \ldots, x_n)$ have no elements on account of the Epsilon Axiom Scheme \ref{Ax:EpsilonScheme}: this is thus a property they share with the empty set. These set matrices $f^n_i(x_1, \ldots, x_n)$ with $n > 1$ are, however, not identical to the empty set $\emptyset$ on account of the Division Axiom Scheme \ref{Ax:DivisionScheme}: these set matrices are objects that are not sets.
\end{Remark}

\begin{Axiom} \label{Ax:Exp-SeparationScheme} (Generalized Axiom Scheme of Separation)\rm: \\
$\forall x \exists y \forall \alpha ( \alpha \in y \Leftrightarrow \alpha \in x \wedge \Phi ( \alpha))$
\end{Axiom}\noindent
This axiom scheme formalizes that for every set $x$ and for every property $\Phi$ there is a subset $y$ of $x$ made up of precisely those elements of $x$ that have the property $\Phi$. Hereby the symbol $\Phi(\alpha)$ represents any well-formed formula with an occurrence of $\alpha$ not bounded by a quantifier $\forall$ or $\exists$. The fact that every well-formed formula has to be finite implies that such a property $\Phi(\alpha)$ can contain only finitely many function symbols $f^n_i$, of which there are infinitely many. And
this implies, that for infinite sets, having elements $f^n_i(x_1, \ldots, x_n)$ for an infinite number of function symbols $f^n_i$, certain properties can not be formulated in a single formula $\Phi$. Of such sets, certain subsets, also having elements $f^n_i(x_1,\ldots, x_n)$ for an infinite number of function symbols $f^n_i$, can thus not be singled out directly by applying the Generalized Axiom Scheme of Separation only once. In section 4.3 this is further elaborated.

\begin{Axiom} \label{Ax:Exp-Pair} (Generalized Pair Axiom)\rm:\\
$\forall \alpha \forall \beta \exists x \forall \gamma (\gamma \in x \Leftrightarrow \gamma = \alpha \vee \gamma = \beta)$
\end{Axiom}\noindent
The Generalized Pair Axiom formalizes that for every pair of matrices $\alpha$ and $\beta$  there is a set $x$ such that the matrices $\alpha$ and $\beta$ are precisely the elements of $x$. From the Generalized Extensionality Axiom for Sets \ref{Ax:Exp-ExtensionalitySets} it follows that this set $x$ is unique, and it can be denoted by $x = \{ \alpha, \beta \}$. Because $1 \times 1$ set matrices $[y]$ and $[z]$ can be taken as value for the matrices $\alpha$ and $\beta$, it follows from the Reduction Axiom \ref{Ax:Reduction} that the Generalized Pair Axiom applies to sets $y$ and $z$, yielding $x = \{ y , z \}$.
\begin{Axiom} \label{Ax:SetOfMatrices} (Set of Matrices Axiom Scheme)\rm:
\begin{flushleft}
$\forall x \exists y (
\forall \alpha_{11} \ldots \forall \alpha_{mn}(\alpha_{11} \in x \wedge \ldots \wedge \alpha_{mn} \in x \Rightarrow
\left [ \begin{array}{ccc} \alpha_{11} & \ldots & \alpha_{1n} \\ \vdots &   & \vdots \\ \alpha_{m1} & \ldots & \alpha_{mn} \end{array} \right ] \in y)
\wedge$
\end{flushleft}
\begin{flushright}
$\forall \beta (
\beta \in y \Rightarrow
\exists \gamma_{11} \ldots \exists \gamma_{mn}
(\beta = \left [ \begin{array}{ccc} \gamma_{11} & \ldots & \gamma_{1n} \\ \vdots &   & \vdots \\ \gamma_{m1} & \ldots & \gamma_{mn} \end{array} \right ] \wedge
\gamma_{11} \in x \wedge \ldots \wedge \gamma_{mn} \in x)))$
\end{flushright}
\end{Axiom}\noindent
The Set of Matrices Axiom Scheme is a countably infinite axiom scheme, with an axiom for every function symbol $f_{m \times n}$. On the one hand, every such axiom guarantees that for every $mn$ elements $\alpha_{11}, \ldots, \alpha_{mn}$ of $x$ there is a matrix $\left [ \begin{array}{ccc} \alpha_{11} & \ldots & \alpha_{1n} \\ \vdots &   & \vdots \\ \alpha_{m1} & \ldots & \alpha_{mn} \end{array} \right ]$ in $y$; on the other hand, it guarantees that there are \emph{no other} elements in $y$ since for every element $\beta$ of $y$ there have to be $mn$ elements $\gamma_{11}, \ldots, \gamma_{mn}$ in $x$ such that $\beta = \left [ \begin{array}{ccc} \gamma_{11} & \ldots & \gamma_{1n} \\ \vdots &   & \vdots \\ \gamma_{m1} & \ldots & \gamma_{mn} \end{array} \right ]$. In every such axiom the set $y$ occurring in it is unique, and can be denoted $M_{m \times n}(x)$, the set of all $m \times n$ matrices with elements of $x$ as entries.

\begin{Axiom} \label{Ax:Exp-SumSet}(Generalized Sum Set Axiom)\rm: \\
$\forall x ( \forall \alpha (\alpha \in x \Rightarrow \exists u (u = \alpha )) \Rightarrow \exists y \forall \beta (\beta \in y \Leftrightarrow \exists z (z \in x \wedge \beta \in z)))$
\end{Axiom}\noindent
The Generalized Sum Set Axiom formalizes that for every set of sets $x$ there is a set $y$ made up precisely of the elements of the sets that are in $x$. Universal quantification over sets of sets is achieved by the restricted quantification $\forall x ( \forall \alpha (\alpha \in x \Rightarrow \exists u (u = \alpha ))\Rightarrow$, because then only sets $x$ are considered such that for every matrix $\alpha$ in $x$ there is a set $u$ identical to that matrix $\alpha$. From the Generalized Extensionality Axiom for Sets \ref{Ax:Exp-ExtensionalitySets} it follows that the set $y$ is unique, and it can be denoted by $y = \bigcup x$.

As an example, consider the set $S$ given by
\begin{equation}
S = \{ \{ \emptyset, [ \emptyset \ \ \emptyset  ] \}, \{ \left [\begin{array} {c} \emptyset \\ \emptyset \end{array} \right ] \} \}
\end{equation}
The set $\bigcup S$ is then the set that contains the elements of the sets $\{ \emptyset, [ \emptyset \ \ \emptyset  ] \}$ and $\{\left [\begin{array} {c} \emptyset \\ \emptyset \end{array} \right ] \}$ in S, and is by axiom \ref{Ax:Exp-SumSet} thus given by:
\begin{equation} \label{eq:14}
\bigcup S = \{\emptyset, [ \emptyset \ \ \emptyset  ], \left [\begin{array} {c} \emptyset \\ \emptyset \end{array} \right ] \}
\end{equation}
This set $\bigcup S$, however, is no longer a set of sets, cf. remark \ref{Rem:MatricesNotSets}. Therefore, it does not follow from the Generalized Sum Axiom that there is a set $\bigcup (\bigcup  S)$. If the restricted quantification in axiom \ref{Ax:Exp-SumSet} would be replaced by a quantification $\forall x$ over sets, then one would have $\bigcup ( \bigcup  S) = \emptyset$ because none of the objects in the set $\bigcup S$ of equation (\ref{eq:14}) has any elements. However, in this axiomatization we have chosen for restricted quantification, because it makes no sense to talk about collecting elements of set matrices, of which by axiom \ref{Ax:EpsilonScheme} is already known that they can't have elements in the sense of the $\in$-relation.

\begin{Remark} \label{Rem:CartesianProduct} \rm
At this point the cartesian product $x \times y$ of two sets $x$ and
$y$ can be introduced:
\begin{equation}
\forall \alpha ( \alpha \in x \times y \Leftrightarrow \exists \beta \exists \gamma ( \beta \in x \wedge \gamma \in y \wedge \alpha = [\beta \ \ \gamma ] ) )
\end{equation}
The set $x \times y$ can then be denoted by $\{[\beta \ \ \gamma ] \ ; \ \beta \in x \wedge \gamma \in y \}$. The existence of the set $x \times y$ is guaranteed by the previous axioms of SMT. Namely, for any two sets $x$ and $y$ the set $\{x, y\}$ exists on account of the Generalized Pair Axiom. The union $x \cup y$ of the sets $x$ and $y$ then exists on account of the Generalized Sum Set Axiom:
\begin{equation}
x \cup y = \bigcup \{x, y \}
\end{equation}
The set $M_{1\times2}(x \cup y)$ then exists on account of the Set of Matrices Axiom Scheme:
\begin{equation}
\begin{array}{lcr}
\forall\alpha \forall \beta (
\alpha \in x \cup y \wedge \beta \in x \cup y \Rightarrow [\alpha \ \ \beta] \in M_{1\times2}(x \cup y))
\wedge \\
\ \ \ \forall \gamma (
\gamma \in M_{1\times2}(x \cup y) \Rightarrow \exists \mu \exists \nu (
\gamma = [\mu \ \ \nu] \wedge \mu \in x \cup y \wedge \nu \in x \cup y))
\end{array}
\end{equation}
The set $x \times y$ then exists on account of the Generalized Axiom Scheme of Separation:
\begin{equation}
\forall\alpha(\alpha \in x \times y \Leftrightarrow \alpha \in M_{1\times2}(x \cup y) \wedge \exists\beta\exists\gamma(\alpha = [\beta \ \ \gamma] \wedge \beta \in x \wedge \gamma \in y))
\end{equation}
It should be noted that because of the Division Axiom Scheme, the sets $x \times y \times z$, $(x \times y) \times z$, and $x \times (y \times z)$, are three mutually different sets; in the framework of ZF, these are also three mutually different sets.
\end{Remark}

\begin{Axiom} \label{Ax:Exp-PowerSet} (Generalized Power Set Axiom)\rm:\\
$\forall x \exists y ( \forall \alpha ( \alpha \in y \Rightarrow \exists u (u = \alpha)) \wedge \forall z (z \in y \Leftrightarrow z \subseteq x))$
\end{Axiom}\noindent
The Generalized Power Set Axiom says that for every set $x$ there is a set of sets $y$ that is made up precisely of the subsets of $x$. This set $y$ is unique on account of the extensionality axiom for sets; notation: \mbox{$y = POW(x)$}. Existential quantification over sets of sets is achieved by the
restricted quantification $\exists y ( \forall \alpha ( \alpha \in y \Rightarrow \exists u (u = \alpha)) \wedge \ldots$

\begin{Definition}\label{Def:SuccessorSet}\rm {\ }\\
Given any set $x$, there is a unique successor set with standard notation $\{x\}$ defined by:\\ $\forall\alpha (\alpha \in \{x\} \Leftrightarrow \alpha = x)$. \\
\end{Definition}\noindent
This definition defines for every set $x$ a singleton $\{x\}$, that has the set $x$ as sole element. The singleton $\{x\}$ is unique on account of the Generalized Extensionality Axiom for Sets \ref{Ax:Exp-ExtensionalitySets}, and its existence is guaranteed on account of the Generalized Pair Axiom \ref{Ax:Exp-Pair}. In the literature, the successor set $\{x\}$ is also denoted by $(x)^+$.

\begin{Axiom} \label{Ax:Exp-Infinity} (Generalized Axiom of Countable Infinity)\rm: \\
$\exists x (\emptyset \in x \wedge \forall y (y \in x \Rightarrow \{y \} \in x))$
\end{Axiom}\noindent

\begin{Remark} \rm
Starting with the empty set $\emptyset$, this axiom thus guarantees
the existence of an infinite set $N$, defined by
\begin{equation} \label{eq:16}
N := \{ \emptyset, \{\emptyset\}, \{\{\emptyset\}\}, \ldots \}
\end{equation}
The elements of $N$ can then be numbered, using $0 := \emptyset$ and $n+1 := \{n\}$, yielding the set of natural numbers $\{0, 1, 2, \ldots \}$. It should be noted that the infinite set $N$, given by (\ref{eq:16}), is not the only infinite set that satisfies the Generalized Axiom of Countable Infinity. For example, the set $\bigcup \{N, \{ \left [ \begin{array} {c} \emptyset \\ \emptyset \end{array} \right ] \} \}$ also satisfies axiom \ref{Ax:Exp-Infinity}. In the framework of ZF, there is also more than one set that satisfies ZF's axiom of infinity.
\end{Remark}

\begin{Axiom} \label{Ax:Exp-SubstitutionScheme} (Generalized Substitution Axiom Scheme)\rm: \\
$\forall x (\forall \alpha(\alpha \in x \Rightarrow \exists!\beta [\Phi(\alpha, \beta)]) \Rightarrow \exists y (\forall \beta (\beta\in y \Leftrightarrow \exists \gamma (\gamma \in x \wedge
\Phi(\gamma, \beta)))))$
\end{Axiom}\noindent
For any set $x$, this axiom formalizes that if every matrix $\alpha$ in $x$ is related to precisely one matrix $\beta$, then there is a set $y$ made up of precisely those matrices $\beta$ that are in relation $\Phi(\gamma, \beta)$ with some matrix $\gamma$ in $x$. The axiom is applicable for any well-formed formula $\Phi(\alpha,\beta)$ that relates every matrix in $x$ with precisely one matrix $\beta$. What has been said about separation applies also here:
the fact that every well-formed formula has to be finite implies that such a formula $\Phi(\alpha, \beta)$ can contain only finitely many function symbols $f^n_i$, of which there are infinitely many. This means that for infinite sets, made up of set matrices $f^n_i(x_1, \ldots, x_n)$ for an infinite number of function symbols $f^n_i$, some
``relations'' would require an infinitely long formula $\Phi(\gamma, \beta)$: the Generalized Substitution Axiom Scheme is then not directly applicable. In section 4.3 this is further elaborated.

\begin{Remark} \label{Rem:FunctionSpace} \rm
At this point the space $y^x$ of all functions from a set $x$ to a
set $y$ can be introduced:\\
$\forall \alpha (\alpha \in y^x \Leftrightarrow \exists f (f := \alpha \wedge f \subseteq x \times y \wedge \forall \beta (\beta \in x \Rightarrow
\exists! \gamma (\gamma \in y \wedge [\beta \ \ \gamma] \in f))))$\\
The set $y^x$ is a subset of the power set of $x \times y$, so that its existence is guaranteed by the Generalized Axiom Scheme of Separation.
\end{Remark}

\begin{Axiom} \label{Ax:Exp-Foundational} (Generalized Foundational Axiom)\rm: \\
$\forall x ( \exists \alpha (\alpha \in x) \wedge \forall \beta (\beta \in x \Rightarrow \exists y (\beta = y)) \Rightarrow
\exists z (z \in x \wedge \forall \gamma (\gamma \in z \Rightarrow \gamma\not \in x)))$
\end{Axiom}\noindent
The Generalized Foundational Axiom formalizes that every nonempty set $x$ of sets has an element $z$, which shares no elements with $x$. This axiom excludes in particular that there is a set $x$ such that $x = \{x\}$.

\begin{Remark}\rm
With the above axiomatization of SMT, there is for every axiom (or axiom scheme) of ZF a corresponding generalized axiom in SMT.
\end{Remark}

\subsection{About the philosophy of mathematics}

From the fact that ZF is adequate as a foundation for virtually all of modern mathematics derives the most widely accepted point of view on what mathematics is: mathematics may be viewed as the body of statements, that can be derived within ZF by means of logical reasoning. Corresponding with this view is the adage `\emph{everything is a set}': every term of every statement is a set -- there are no other terms. Having defined the framework of SMT, the general philosophy of what mathematics is may then be distilled from this view: mathematics is the body of statements that can be derived within the framework of SMT by means of formal deduction.

In the framework of SMT, besides sets also set matrices occur as terms of the mathematical language. These set matrices are in general not sets, so that the adage `\emph{everything is a set}' of ZF is certainly not valid in the framework of SMT. However, because of the Reduction Axiom \ref{Ax:Reduction}, in the framework of SMT every set $x$ is identical to a $1 \times 1$ set matrix $[x]$. As a result, the adage `\emph{everything is a matrix}' holds in the framework of SMT. Concerning the terms of the language a nominalist position is taken, in the sense that these terms (sets and matrices of sets) in themselves are without any fundament in physical reality -- that is, there is no Platonian domain in reality that \emph{is} the universe of SMT.

Furthermore, given that the motivation for the development of SMT lies in physics, the point of view on the position of mathematics in the whole of science is reflected by the adage `\emph{mathematica ancilla physicae}'. This implies the view that mathematics in the first place is meant to provide a language for the natural sciences.

\section{Discussion}
\subsection{The relation between SMT and ZF}
\begin{Remark}\rm
An undeniable observation is that the language $\rm L_{\it ZF}$ of ZF is properly contained in the language $\rm L_{\it SMT}$ of SMT.  Now let L' be the restriction of the language $\rm L_{\it SMT}$ by leaving out the function symbols $f^n_i$  with $n > 1$, and let $\rm SMT|L'$ be the restriction of SMT to the language L'.  Thus speaking, the only function symbol in L' is $f^1_1$  with $f^1_1(x) = [x]$, and '=' and '$\in$' are the only predicate letters in L', so that the Omission Axiom Scheme \ref{Ax:OmissionScheme}, the Epsilon Axiom Scheme \ref{Ax:EpsilonScheme}, the Division Axiom Scheme \ref{Ax:DivisionScheme}, and the Extensionality Axiom Scheme for Set Matrices \ref{Ax:ExtensionalityMatrices} do not occur in $\rm SMT|L'$.
\end{Remark}

\begin{Theorem} \label{Th:SetsAreMatrices} {\ }
$\vdash _{\rm SMT|L'} \forall x \exists! \alpha (\alpha = [x]) \wedge \forall \alpha \exists! x (x = \alpha)$
\end{Theorem}
\begin{proof}
In $\rm SMT|L'$ the Set Matrix Axiom Scheme \ref{Ax:SetMatrixScheme} is reduced to
\begin{equation}\label{eq:21}
\forall x \exists \alpha (\alpha = [x])
\end{equation}
Because of the absence of the other function symbols $f^n_i$ in $\rm SMT|L'$, the variables ranging over matrices only range over these $1 \times 1$ set matrices $[x]$. Using this completeness argument and the Reduction Axiom it then follows that
\begin{equation}\label{eq:22}
\vdash _{\rm SMT|L'} \forall \alpha \exists x (x = \alpha)
\end{equation}
Uniqueness in (\ref{eq:21}) and (\ref{eq:22}) follows from symmetry and transitivity of the identity relation.
\end{proof}

\begin{Theorem}\label{Th:Quantification} {\ }
$\vdash _{\rm SMT|L'} \forall \alpha (\Phi(\alpha)) \Leftrightarrow \forall x (\Phi(x))$
\end{Theorem}
\begin{proof}
This follows from theorem \ref{Th:SetsAreMatrices} and substitutivity of equality. This means that in $\rm SMT|L'$ quantification over all matrices is equivalent to quantification over all sets.
\end{proof}

\begin{Theorem} \label{Th:ElementsAreSets}
$\vdash _{\rm SMT|L'} \forall x \forall \alpha (\alpha \in x \Rightarrow \exists y (y = \alpha))$
\end{Theorem}
\begin{proof}
It has been demonstrated that $\vdash _{\rm SMT|L'} \forall \alpha \exists! x (x = \alpha)$ in theorem \ref{Th:SetsAreMatrices}. So in particular, expression \ref{Th:ElementsAreSets} follows.
\end{proof}


\begin{Proposition}(Relation between SMT and ZF)\rm:\\
The restriction of SMT to L' is a conservative extension of ZF.
\end{Proposition}
\begin{proof}
Using theorems \ref{Th:Quantification} and \ref{Th:ElementsAreSets}, one can easily prove that $\vdash _{\rm SMT|L'}A$ for every axiom $A$ of ZF. In other words, one can easily prove for every formula $\Psi$  of ZF that if $\vdash _{ZF} \Psi$, then $\vdash _{\rm SMT|L'}\Psi$. This shows that $\rm SMT|L'$ is an \emph{extension} of ZF. In addition, using theorem \ref{Th:Quantification} one can easily prove that $\rm SMT|L'$ is in fact equivalent to ZF in its original formulation plus the two new axioms (the one remaining axiom (\ref{eq:21}) of the Set Matrix Axiom Scheme \ref{Ax:SetMatrixScheme} and the Reduction Axiom \ref{Ax:Reduction}) for the new constants -- the $1 \times 1$ set matrices $[x]$. This shows that $\rm SMT|L'$ is a \emph{conservative} extension of ZF.
\end{proof}

\begin{Remark}\rm
The previous proposition proves that every set that can be constructed in ZF, can also be constructed in SMT. It should be noted, however, that (unrestricted) SMT is not an extension of ZF in the accepted sense of the word `extension', cf. \cite{bib:Shoenfield}. That is, it is not the case that every theorem of ZF is a theorem of SMT. For example, it is a theorem of ZF that there is precisely one set which has no sets as elements:
\begin{equation}
\vdash _{ZF}\exists! x \forall y (y \not \in x)
\end{equation}
In SMT, this theorem does not hold. For example, the set $\{ \left [ \begin{array}{c} \emptyset \\ \emptyset \end{array} \right ]\}$ has no sets as elements (because its element $\left [ \begin{array}{c} \emptyset \\ \emptyset \end{array} \right ]$ is not a set), but the set $\{ \left [ \begin{array}{c} \emptyset \\ \emptyset \end{array} \right ]\}$ is not identical to the empty set on account of the Generalized Extensionality Axiom for Sets \ref{Ax:Exp-ExtensionalitySets}. Thus, in SMT there are at least two sets that have no sets as elements, which proves that the aforementioned formula  $\exists! x \forall y (y \not \in x)$ does not hold in SMT. Thus, SMT is \emph{not} an extension of ZF.
\end{Remark}

\begin{Remark}\rm
Instead, SMT is to be viewed as a \emph{generalization} of ZF. A suggestion for a definition of this notion is the following: a theory $T'$ (in this case: SMT) is a generalization of a theory $T$ (in this case: ZF) if and only if the following conditions are satisfied:
\begin{enumerate}[(i)]
\item the language $L_{T'}$ for $T'$ is a proper extension of the language $L_T$ of $T$;
\item the universe for $L_{T'}$ properly contains the universe for $L_{T}$;
\item there is a language $L'$ (in our case $\rm L_{\it SMT}$ without the function symbols $f^n_i$ with $n > 1$), such that $L'$ is an extension of $L_T$ and $L_{T'}$ is an extension of $L'$, and such that the restriction of $T'$ to $L'$ is a conservative extension of $T$.
\end{enumerate}
These conditions are precisely satisfied in our case.
\end{Remark}

\subsection{Resolving the issues with the formalization of the EPT}

The main reason for introducing SMT is that it is more suitable as a mathematical foundation for the EPT than ZF. The complications that arise from a formalization of the EPT in the framework of ZF have been discussed in section 1. In this section it will be shown that these complications do not exist in the framework of SMT.

First of all, in the framework of SMT the $2 \times 1$ set matrices $\left [ \begin{array}{c} x \\ y \end{array} \right ]$ exist as such in the mathematical universe; on account of the Division Axiom Scheme these set matrices are not identical to any set. Therefore, such $2 \times 1$ set matrices can be used as \emph{direct} designators of components of the physical universe, consisting of a constituent of the world (designated by the entry $x$ in the first row) and a constituent of the antiworld (designated by the entry $y$ in the second row). Thus speaking, by taking SMT as the mathematical foundation for the EPT, the demand for the truth condition of knowledge can be met that entities that occur in the ontology for physics are to be designated by entities that occur as such in the ontology for mathematics: the complication that arises from a formalization of the EPT in the framework of ZF is thus absent in the framework of SMT.

Next, expressions of the type
$\left [ \begin{array}{c} a \\ b \end{array} \right ]: \left [ \begin{array}{c} f \\ g \end{array} \right ] \begin{array}{c} \rightarrow \\ \leftarrow \end{array} \left [ \begin{array}{c} x \\ y \end{array} \right ]$ can be formalized in the framework of SMT as a standard notation for a ternary relation $R$:
\begin{equation}\label{eq:24}
\left [ \begin{array}{c} a \\ b \end{array} \right ]:
\left [ \begin{array}{c} f \\ g \end{array} \right ]
\begin{array}{c} \rightarrow \\ \leftarrow \end{array}
\left [ \begin{array}{c} x \\ y \end{array} \right ]
\Leftrightarrow
\langle  \left [ \begin{array}{c} a \\ b \end{array} \right ], \left [ \begin{array}{c} f \\ g \end{array} \right ], \left [ \begin{array}{c} x \\ y \end{array} \right ] \rangle \in R
\end{equation}
Such a formalization is also possible in ZF, but the point is that the $2 \times 1$ set matrices $ \left [ \begin{array}{c} a \\ b \end{array} \right ]$, $\left [ \begin{array}{c} f \\ g \end{array} \right ]$, and $\left [ \begin{array}{c} x \\ y \end{array} \right ]$ are not identical to sets $S$, $T$, and $V$ in the framework of SMT because of the Division Axiom Scheme. Thus, in the framework of SMT one gets\footnote{Suppose that $S:T\rightleftarrows V$ can be deduced. Given expression (\ref{eq:24}), this would imply that $\langle S, T, V\rangle \in R$. But the set $R$ is well defined in \cite{bib:Cabb}: it only contains three-tuples of the form $\langle  \left [ \begin{array}{c} a \\ b \end{array} \right ], \left [ \begin{array}{c} f \\ g \end{array} \right ], \left [ \begin{array}{c} x \\ y \end{array} \right ] \rangle$. And given the Division Axiom Scheme \ref{Ax:DivisionScheme}, such a three-tuple is not identical to a three-tuple $\langle S, V, T\rangle$. Thus, $\langle S, V, T\rangle$ is not in $R$. Contradiction. Thus, $S:T\rightleftarrows V$ \emph{cannot} be deduced from the axioms of the EPT in the framework of SMT.}
\begin{equation}\label{eq:24a}
\left [ \begin{array}{c} a \\ b \end{array} \right ]:
\left [ \begin{array}{c} f \\ g \end{array} \right ]
\begin{array}{c} \rightarrow \\ \leftarrow \end{array}
\left [ \begin{array}{c} x \\ y \end{array} \right ]
\nvdash_{SMT}
S:T\rightleftarrows V
\end{equation}
Clearly, expression (\ref{eq:24a}) is in contrast with expression (\ref{eq:6a}): if SMT is thus taken as the mathematical foundation for the EPT, then the axiomatic system containing the EPT has no theorems of the type $S:T\rightleftarrows V$, which can \emph{not} be translated into a statement about the physical universe by the interpretation rules and are thus physically uninterpretable. This shows that the second complication that arises from a formalization of the EPT in the framework of ZF, is absent in the framework of SMT.

\subsection{Resolving infinities arising in separation and substitution}

\begin{Example}\label{Ex:InfinitySeparation}\rm
Let the set $N^* = N - \{0\} = \{1, 2, 3, \ldots \}$. Consider the sets $M_{m \times n}(N)$ given by
\begin{equation}\label{eq:25}
M_{m \times n}(N) = \{
    \left [ \begin{array}{ccc} x_{11} & \ldots & x_{1n} \\ \vdots &   & \vdots \\ x_{m1} & \ldots & x_{mn}\end{array} \right ]
    \ ; \ x_{11}, \ldots, x_{mn} \in N \}
\end{equation}
Let $\{M_{m \times n}(N) \ ; \ m, n \in N^* \}$ be the set of all these sets $M_{m \times n}(N)$, and let the set\linebreak $S = \bigcup_{m, n \in N^*}\{M_{m \times n}(N)\}$. In words, $S$ is the set of all matrices of all dimensions $m \times n$ with entries from the set of natural numbers $N$. It is then not possible to single out the subset $T$ of all matrices of all dimensions $m \times n$ with entries from the set $2N +1 = \{1, 3, 5, \ldots \}$ of odd natural numbers, by applying the Generalized Axiom Scheme of Separation only once to the set $S$. Namely, the formula
\begin{equation}
\begin{array}{lcr}
\exists y (\left [ \begin{array}{ccc} x_{11} & \ldots & x_{1n} \\ \vdots &   & \vdots \\ x_{m1} & \ldots & x_{mn}\end{array} \right ] \in y \Leftrightarrow\\
\ \ \ \left [ \begin{array}{ccc} x_{11} & \ldots & x_{1n} \\ \vdots &   & \vdots \\ x_{m1} & \ldots & x_{mn}\end{array} \right ] \in S \wedge
    x_{11}, \ldots, x_{mn} \in 2N+1)
\end{array}
\end{equation}
is not an instance of the Generalized Axiom Scheme of Separation if $m$ and $n$ are undetermined, because the term
$\left [ \begin{array}{ccc} x_{11} & \ldots & x_{1n} \\ \vdots &   & \vdots \\ x_{m1} & \ldots & x_{mn}\end{array} \right ]$
then can not be constructed using the clauses of definition \ref{Def:Syntax}. However, for every set $M_{m \times n}(N)$ defined by (\ref{eq:25}) the formula
\begin{equation}
\begin{array}{lcr}
\exists y(\left [ \begin{array}{ccc} x_{11} & \ldots & x_{1n} \\ \vdots &   & \vdots \\ x_{m1} & \ldots & x_{mn}\end{array} \right ] \in y \Leftrightarrow \\
\ \ \ \left [ \begin{array}{ccc} x_{11} & \ldots & x_{1n} \\ \vdots &   & \vdots \\ x_{m1} & \ldots & x_{mn}\end{array} \right ] \in M_{m \times n}(N) \wedge
    x_{11}, \ldots, x_{mn} \in 2N+1)
\end{array}
\end{equation}
is a well-formed formula for fixed $m$ and $n$, if $x_{11}, \ldots, x_{mn} \in 2N+1$ is as usual viewed as an abbreviation of $x_{11} \in 2N+1 \wedge \ldots \wedge x_{mn} \in 2N+1$. The set $y$ is then unique, and \linebreak $y = M_{m \times n}(2N+1)$. The requested subset $T$ of $S$ is then defined by
\begin{equation}
T = \bigcup_{m, n \in N^*}\{M_{m \times n}(2N+1)\}
\end{equation}
It thus takes infinitely many applications of the Generalized Axiom Scheme of Separation to construct the set $T$ as a subset of $S$.\hfill $\square$
\end{Example}

From the previous example it \emph{seems} that the Generalized Axiom Scheme of Separation has less power in the framework of SMT than its counterpart has in the framework of ZF, due to the occurrence of the metavariable $\Phi$. The point is that $\Phi$ represents a first order formula, and thus has to be finite: for a set with infinitely many different types of matrices as elements, it thus becomes impossible to formulate certain properties covering all elements in a single finite formula, purely because the starting set is made up of infinitely many different types of matrices.

The next proposition shows that this seeming loss of power does not lead to an incompleteness in the framework of SMT: it is merely the case that in the framework of SMT the Generalized Axiom of Separation has to be applied infinitely many times to single out certain subsets of sets, made up of infinitely many different types of set matrices. But that can still be done \emph{within} the framework.

\begin{Proposition}\label{Pr:InfiniteSeparationScheme}\rm Let $x$ be a set, made up of infinitely many different types of set matrices. Then any subset $y \subset x$, the construction of which would require an instance $\Phi(f^{n}_{i}(x_1, \ldots, x_n))$ of a well-formed formula $\Phi(\alpha)$ for infinitely many function symbols $f^n_i$, is constructible in the framework of SMT.
\end{Proposition}
\begin{proof}
To start with, for any set $x$, subsets $x^n_i$ of $x$ can be singled out on account of the Generalized Axiom Scheme of Separation:
\begin{equation}
\forall x \exists x^1_1 \forall \alpha (\alpha \in x^1_1 \Leftrightarrow \alpha \in x \wedge \exists x_1 (\alpha = [x_1]))
\end{equation}
\begin{equation}
\forall x \exists x^2_1 \forall \alpha (\alpha \in x^2_1 \Leftrightarrow \alpha \in x \wedge \exists x_1 \exists x_2 (\alpha = [x_1 \ \ x_2]))
\end{equation}
\begin{equation}
\forall x \exists x^2_2 \forall \alpha (\alpha \in x^2_2 \Leftrightarrow \alpha \in x \wedge \exists x_1 \exists x_2 (\alpha = \left [ \begin{array}{c} x_1 \\ x_2 \end{array} \right ]))
\end{equation}
and so forth. So, $x^1_1$ is the set of all sets in $x$, $x^2_1$ is the set of all matrices in $x$ of the form $[x_1 \ \ x_2]$, $x^2_2$ is the set of all matrices in $x$ of the form $\left [ \begin{array}{c} x_1 \\ x_2 \end{array} \right ]$, etc. Because the list of function symbols $f^n_i$ is complete, cf. remark \ref{Rem:ListExhaustive}, every element $\alpha$ of $x$ is in at least one such subset $x^n_i$, and because of the Division Axiom Scheme \ref{Ax:DivisionScheme} every element $\alpha$ of $x$ is in at most one such subset $x^n_i$.

Next, for every such subset $x^n_i$ of $x$ a subset $y^n_i$ can be singled out on account of the Generalized Axiom Scheme of Separation:
\begin{equation}
\forall x_1 \ldots \forall x_n (f^{n}_{i}(x_1, \ldots, x_n) \in y^n_i \Leftrightarrow f^{n}_{i}(x_1, \ldots, x_n) \in x^n_i \wedge \Phi(f^{n}_{i}(x_1, \ldots, x_n)))
\end{equation}
The point here is that the elements of $x^n_i$ all use the same function symbol $f^n_i$ so that a term $f^{n}_{i}(x_1, \ldots, x_n)$ occurring in $\Phi$ ranges over all elements of $x^n_i$.

Finally, let $I$ be the set of all indices $\langle n, i \rangle$ that occur in function symbols $f^n_i$; then for each element $\langle n, i \rangle$ of $I$ there is precisely one set $y^n_i$. On account of the Generalized Substitution Axiom Scheme \ref{Ax:Exp-SubstitutionScheme} there is thus a set $\{y^n_i | \langle n, i \rangle \in I \}$ made up of these sets $y^n_i$. On account of the Generalized Sumset Axiom there is thus a set $y$ that satisfies
\begin{equation}
y = \bigcup \{y^n_i | \langle n, i \rangle \in I \}
\end{equation}
This is precisely the set $y$ requested: this proves the proposition.
\end{proof}

Similarly, also with the Generalized Substitution Axiom Scheme \ref{Ax:Exp-SubstitutionScheme} a \emph{seeming} loss of power is connected, again because of the occurrence of the metavariable $\Phi$ for formulas. As the next example \ref{Ex:InfinitySubstitution} shows, for a set made up of infinitely many different types of matrices some functional ``relations'' would require infinitely many first-order formula $\Phi$, purely because the set is made up of infinitely many different types of matrices. The Generalized Substitution Axiom Scheme is then not directly applicable to construct the image of such a set under a function. However, as proposition \ref{Pr:InfiniteFunctionalRelation} will demonstrate, there is no incompleteness involved with the Generalized Substitution Axiom Scheme: an infinite scheme similar to that in Proposition \ref{Pr:InfiniteSeparationScheme} can be applied.

\begin{Example}\label{Ex:InfinitySubstitution}\rm
Consider once more the set $S$ of all matrices that can be constructed from the set $N$, cf. example \ref{Ex:InfinitySeparation}. Now for each matrix $\alpha$ in $S$, for which $\alpha = \left [ \begin{array}{ccc} t_{11} & \ldots & t_{1n} \\ \vdots &   & \vdots \\ t_{m1} & \ldots & t_{mn} \end{array} \right ]$, there is exactly one matrix $\beta = 2\alpha$ in $S$ for which $\beta = 2\cdot \left [ \begin{array}{ccc} t_{11} & \ldots & t_{1n} \\ \vdots &   & \vdots \\ t_{m1} & \ldots & t_{mn} \end{array} \right ] = \left [ \begin{array}{ccc} 2\cdot t_{11} & \ldots & 2\cdot t_{1n} \\ \vdots &   & \vdots \\ 2\cdot t_{m1} & \ldots & 2\cdot t_{mn} \end{array} \right ]$ where $2\cdot \left [ \begin{array}{ccc} x_{11} & \ldots & x_{1n} \\ \vdots &   & \vdots \\ x_{m1} & \ldots & x_{mn} \end{array} \right ] = \left [ \begin{array}{ccc} 2x_{11} & \ldots & 2x_{1n} \\ \vdots &   & \vdots \\ 2x_{m1} & \ldots & 2x_{mn} \end{array} \right ]$ for $x_{11}, \ldots, x_{mn} \in N$.
This functional relation, however, cannot be formalized in a finite formula $\Phi$, because there are infinitely many different types of matrices $\alpha$. Namely, the expression
\begin{equation}
\forall \alpha (\alpha = f^{n}_{i}(x_1, \ldots, x_n) \wedge x_{1} \in N \wedge \ldots \wedge x_{n} \in N \Rightarrow \exists! \beta (\beta = 2\alpha = f^{n}_{i}(2x_1, \ldots, 2x_n)))
\end{equation}
is not a well-formed formula because $n$ and $i$ are undefined: this expression can thus not appear in an instance of the Generalized Substitution Axiom Scheme. \hfill $\Box$
\end{Example}

\begin{Proposition}\label{Pr:InfiniteFunctionalRelation}\rm
Let $x$ be set, made up of infinitely many different types of set matrices. Then any image set $y$, the construction of which would require an instance \linebreak $\Phi(f^{n}_{i}(x_1, \ldots, x_n), f^{m}_{j}(y_1, \ldots, y_m))$ of a well-formed functional relation $\forall \alpha \in x \exists! \beta \Phi(\alpha, \beta)$ for infinitely many combinations of function symbols $\langle f^n_i, f^m_j \rangle$ (with at most one such combination for each function symbol $f^n_i$), is constructible in the framework of SMT.
\end{Proposition}
\begin{proof}
To start with, for any set $x$, subsets $x^n_i$ of $x$ can be singled out on account of the Generalized Axiom Scheme of Separation; see proposition \ref{Pr:InfiniteSeparationScheme}. Next, for every such subset $x^n_i$ of $x$ an image set $y^n_i$ can be constructed on account of the Generalized Substitution Axiom Scheme: this set $y^n_i$ is the collection of matrices $f^{m}_{j}(y_1, \ldots, y_m)$, for which there is an element $f^{n}_{i}(x_1, \ldots, x_n) \in x^n_i$ such that $\Phi(f^{n}_{i}(x_1, \ldots, x_n), f^{m}_{j}(y_1, \ldots, y_m))$ -- note that $\Phi$ is assumed to be a functional relation. Again on account of the Generalized Substitution Axiom Scheme \ref{Ax:Exp-SubstitutionScheme} there is thus a set $\{y^n_i | \langle n, i \rangle \in I \}$ made up of these sets $y^n_i$, where the set $I$ is defined as in proposition \ref{Pr:InfiniteSeparationScheme}. On account of the Generalized Sumset Axiom there is thus a set $y$ that satisfies
\begin{equation}
y = \bigcup \{y^n_i | \langle n, i \rangle \in I \}
\end{equation}
This is precisely the image set $y$ requested: this proves the proposition.
\end{proof}

\section{Conclusions}

The first conclusion is that SMT is better suited than ZF as a foundation for mathematics in the research program on the EPT. It has been shown that a formalization of the EPT in the framework of ZF leads to unsolvable complications such as weakly surplus structure, and it has been demonstrated that SMT resolves these complications by expanding the ontological repertoire of the language of mathematics in the sense that set matrices exist as such in the mathematical universe of SMT. 

The second conclusion is that SMT is not better suited than ZF as a foundation for mathematics from the point of view of pure mathematics. It is true that SMT is not weaker than ZF, as is demonstrated by the fact that a restriction of SMT is a conservative extension of ZF: all sets that can be constructed with ZF can thus also be constructed with SMT. It is also true that SMT yields an incremental improvement, because \emph{n}-ary structures fit more elegantly in the ontology corresponding with the framework of SMT than in the ontology corresponding with the framework of ZF. That is, binary structures such as groups and topological spaces are simply $1 \times 2$ set matrices in the framework of SMT, ternary structures such as fields are $1 \times 3$ set matrices, etc.: it is not the case that these structures can \emph{not} be represented in the framework of ZF, but their representation is less elegant. A group, for example, is a two-tuple $\langle G, * \rangle$ so that strictly speaking a group is a set $\{\{G\}, \{G, *\}\}$ in the framework of ZF; a $1 \times 2$ set matrix $[G \ \ *]$ is then a more elegant representation of a group than this set $\{\{G\}, \{G, *\}\}$. The crux, however, is that SMT does not make the foundation of mathematics in itself more powerful: the basic set-theoretical questions, that are unsolvable in ZF, remain unsolved in SMT.

The bottom line is that SMT is \emph{preferred} over ZF in the research program on the EPT. The motivation for this preference is also used elsewhere: Teller argues in \cite{bib:Teller} that Fock space is preferred as the Hilbert space for Quantum Field Theory, because it eliminates the problem of weakly surplus structure that arises when a tensor product of one-quantum Hilbert spaces is used to describe two-particle states. Currently the scope of the present results is limited to the research program on the EPT, but as it concerns the very foundations of mathematics, the inappropriateness of ZF for the formalization of the EPT may have implications beyond its current limit in the eventuality that the research program on the EPT supersedes the other research programs in physics.

\paragraph{Acknowledgements.}
We wish to thank Sergey Sannikov, Institute of Theoretical Physics, Kharkov Institute of Physics and Technology (deceased the $25^{\rm th}$ of March, 2007), for his contribution to the development of SMT. We are also grateful to two anonymous referees who by their critical remarks enabled us to improve the paper considerably.
This work was financially supported by PlusPunt Eindhoven B.V., grant $\#823353$. Additional financial support was given by the Foundation Liberalitas.

\bibliographystyle{abbrv}
\bibliography{main}

\begin{thebibliography}{10}
\bibitem{bib:Cabb}
{\sc Cabbolet},~M.J.T.F.: Elementary Process Theory: a Formal Axiomatic System With a Potential Application as a Foundational Framework for Physics Supporting Gravitational Repulsion of Matter and Antimatter, \emph{Ann. Phys. (Berlin)}, \textbf{522}(10), 699-738 (2010)
\bibitem{bib:Cabb2}
{\sc Cabbolet},~M.J.T.F.: Addendum to the Elementary Process Theory, \emph{Ann. Phys. (Berlin)}, \textbf{523}(12), 990-994 (2011)
\bibitem{bib:Cabb3}
{\sc Cabbolet},~M.J.T.F.: Elementary Process Theory: axiomatic introduction and applications. Dissertation, Vrije Universiteit Brussel (2011)
\bibitem{bib:Feyn}
{\sc Feynman},~P.R.: The Theory of Positrons, \emph{Phys. Rev.}, \textbf{76}, 749-759 (1949)
\bibitem{bib:Dalen}
{\sc Dalen},~D.~van, H.C.~{\sc Doets}, and H.C.M.~{\sc de~Swart}:
Sets: Naive, Axiomatic and Applied. Pergamon Press, Oxford (1978)
\bibitem{bib:Redhead}
{\sc Redhead},~M.L.G.: Symmetry in Intertheory Relations, \emph{Synthese} \textbf{32}, 77-112 (1975)
\bibitem{bib:Shoenfield}
{\sc Shoenfield},~J.R.: Mathematical Logic, AK Peters Ltd., Natick (2001)
\bibitem{bib:Teller}
{\sc Teller},~P.: An Interpretive Introduction to Quantum Field Theory. Princeton University Press, Princeton (1997)
\end{thebibliography}

\end{document}